\documentclass[a4paper, 11pt]{amsart}
\usepackage{amsthm}
\usepackage{tikz-cd}
\usepackage{algorithm, algpseudocode, dsfont, color, soul, amsmath, amssymb, amsfonts, amsthm, bbm, bbold, fixmath, mathtools, multirow, boldline, xcolor, colortbl, url, footnote, cite}

\usepackage[top=3cm, bottom=3cm, left = 3cm, right = 3cm]{geometry} 
\usepackage[utf8]{inputenc}
\usepackage{textcomp}
\usepackage{graphicx} 
\usepackage{amsmath,amssymb} 
\usepackage{amsthm}
\usepackage{setspace} 

\newtheorem{theorem}{Theorem}[section]
\newtheorem{theorem*}{Theorem}[section]
\newtheorem{corollary}[theorem]{Corollary}
\newtheorem{lemma}[theorem]{Lemma}

\newtheorem{proposition}[theorem]{Proposition}
\theoremstyle{definition}
\newtheorem{definition}[theorem]{Definition}

\usepackage{mathtools}
\usepackage[pdftex,bookmarks,colorlinks,breaklinks]{hyperref}  
\usepackage{memhfixc} 
\usepackage{pdfsync}  
\usepackage{fancyhdr}

\usepackage{nccmath}
\usepackage{hyperref}

\usepackage{cleveref}
\newtheorem*{d2}{The D2 Problem}
\newtheorem*{grp*}{Geometric realisation problem}
\newtheorem{question}{Question}

\newtheorem{thmx}{Theorem}
\theoremstyle{remark}
\newtheorem{remark}{Remark}

\title{Geometric realisation over aspherical groups}
\author{William Thomas}
\address{Department of Mathematics, Imperial College London, London, SW7 2AZ, U.K.}
\email{wjt220@ic.ac.uk}

\begin{document}
\maketitle

\begin{abstract}
    We prove that the direct sums of extensions of scalars of relation modules are geometrically realisable as the second homotopy group of a finite 2-complex. We use this to exhibit a finite 2-complex with fundamental group the $(10,15)$ torus knot group and non-free $\pi_2$, yielding exotic presentations of a group for which no such examples had previously been known. We conclude by constructing stably free non-free modules over an infinite family of Baumslag-Solitar groups; it remains to determine whether these modules are geometrically realisable by finite 2-complexes.
\end{abstract}

\section{Introduction}

Given a finitely presented group $G$ and finite $2$-complex $X$ with $\pi_1(X)\cong G$ and universal cover $\tilde{X}$, we can consider the following exact sequence of $\mathbb{Z}G$-modules obtained from the cellular chain complex for $\tilde{X}$:

$$ C_\ast(\tilde{X})=(C_2(\tilde{X})\xrightarrow{\partial_2} C_1(\tilde{X})\xrightarrow{\partial_1} C_0(\tilde{X})\xrightarrow{\epsilon}\mathbb{Z}\to 0)$$ where the action of $\mathbb{Z}G$ is induced by the monodromy action of $\pi_1(X)$ on $\tilde{X}$. This chain complex, when considered up to chain homotopy equivalence and the action of $Aut(G)$, is in fact a homotopy invariant for the space $X$. 

For the remainder of this article we assume groups $G$ to be finitely presented. Consider an exact sequence of the following form: $$A=(A_2\xrightarrow{\partial_2} A_1\xrightarrow{\partial_1} A_0\xrightarrow{\epsilon}\mathbb{Z}\to 0)$$ where the $A_i$ are finitely-generated free $\mathbb{Z}G$-modules, we call this an \textit{algebraic $(G,2)$-complex}.

\begin{definition}
    Given an algebraic $(G,2)$-complex $A$, we say that $A$ is \textit{geometrically realisable} if there exists a finite $2$-complex $X$ such that $\pi_1(X)\cong G$ and $C_\ast(\tilde{X})$ is chain-homotopy equivalent to $A$.
\end{definition}

We are motivated by the following important open problem in low-dimensional topology, which we relate to another famous problem in Section \ref{grp}:

\begin{grp*} Is every algebraic $(G,2)$-complex geometrically realisable?
\end{grp*}

In this paper we exclusively study geometric realisation over \textit{aspherical} groups, a large class of torsion-free groups that includes all (torsion-free) 1-relator groups, see Definition \ref{ad}. Later we will see that for $G$ aspherical we have the following definition of geometric realization of stably-free $\mathbb{Z}G$-modules that yields an equivalent formulation of the Geometric realisation problem:

\begin{definition}
      Let $G$ be a group and $M\in SF(\mathbb{Z}G)$, we say that $M$ is `geometrically realisable' if there is a finite 2-complex $X$ such that $\pi_1(X)\cong G$ and $\pi_2(X)\cong M$. 
\end{definition}

 Relation modules, see Definition \ref{rm}, are an important class of stably-free modules over aspherical groups, that were shown to be geometrically realisable by Harlander-Jenson in \cite{MR2263062}. The majority of non-trivial examples of geometrically realisable stably-free modules come from this source, for example in Berridge-Dunwoody \cite{MR540056}, Preusser \cite{MR263} and \cite{MR2263062}. We provide the following extension to that key result:  

\begin{thmx}
    
\label{1}
    If $H_1,...H_k\leqslant G$ for $k\geqslant 1$ are aspherical groups and $M_i\in SF(\mathbb{Z}H_i)$ for $i=1,...,k$ are relation modules, then \[(\mathbb{Z}G\otimes_{\mathbb{Z}H_1} M_1)\oplus (\mathbb{Z}G\otimes_{\mathbb{Z}H_2} M_2)\oplus...\oplus(\mathbb{Z}G\otimes_{\mathbb{Z}H_k} M_k)\in SF(\mathbb{Z}G)\] is geometrically realisable.

\end{thmx}

As corollaries to this we get the following special cases: 

\begin{corollary}\label{3}
    If $G$ is an aspherical group and $M_1, M_2,..., M_k\in SF(\mathbb{Z}G)$ for $k\geqslant 1$ are relation modules, $M_1\oplus M_2\oplus...\oplus M_k$ is geometrically realisable.
\end{corollary}
\begin{proof}
    Let $H_i=G$ for $i=1,...,k$ in Theorem \ref{1}.
\end{proof}

\begin{corollary}\label{4}
    If $H\leqslant G$ are aspherical groups and $M\in SF(\mathbb{Z}G)$ is a relation module, then $\mathbb{Z}G\otimes_{\mathbb{Z}H} M$ is geometrically realisable.
\end{corollary}
\begin{proof}
    Let $k=1$ in Theorem \ref{1}.
\end{proof}

We prove this theorem as a corollary to the following more general result:

\begin{thmx}
\label{2}
      If $H_1,...H_k\leqslant G$ for $k\geqslant 1$ are aspherical groups and for $i=1,...,k$ $N_i\in SF(\mathbb{Z}H_i)$ are geometrically realisable by finite $2$-complexes $K_i$ containing aspherical subcomplexes $L_i$ with fundamental group $H_i$, then \[(\mathbb{Z}G\otimes_{\mathbb{Z}H_1} N_1)\oplus (\mathbb{Z}G\otimes_{\mathbb{Z}H_2} N_2)\oplus...\oplus(\mathbb{Z}G\otimes_{\mathbb{Z}H_k} N_k)\in SF(\mathbb{Z}G)\] is geometrically realisable.
    
\end{thmx}

\begin{remark}
    The assumption that $H\leqslant G$ are both aspherical is natural, as $G$ aspherical implies $H$ aspherical if $[G:H]$ is finite by taking the galois covering of the aspherical 2-complex for $G$ corresponding to $H$. We will also exhibit non-trivial examples of this phenomena without this index assumption, in particular $T(2,3)\leqslant T(2n,3n)$ for any $n\in\mathbb{N}$ are both aspherical as are torsion-free 1-relator groups.
\end{remark}

We will then use Corollary \ref{4} to show the following, where a pair of presentations $\mathcal{P},\mathcal{Q}$ for a group $G$ is called \textit{exotic} if they have the same deficiency but their presentation complexes $\mathcal{X}_\mathcal{P}, \mathcal{X}_\mathcal{Q}$ are not homotopy equivalent, see \cite[p6555]{MR4302168}. We let $T(n,m)$ for $n,m\in\mathbb{Z}$ be the (generalised) torus knot group with standard presentation $$T(n,m)=\langle x,y\mid x^n=y^m\rangle$$

\begin{thmx}\label{9}
     There exist exotic presentations $\mathcal{P},\mathcal{Q}$ for the group $T(10,15)$ given by \[\mathcal{P}=\langle a,b \mid a^{10}=b^{15}, 1\rangle\]    \[\mathcal{Q}=\langle p,q, p',q'\mid p^{10}=q^{15}, p'^{2}=q'^3, p^{15}=p'^3, q^{20}=q'^4\rangle\] or equivalently there is a non-free stably-free $\mathbb{Z}T(10,15)$-module that is geometrically realisable.
\end{thmx}

\begin{remark}
    We believe that the above Theorem \ref{9} can be extended to all torus knot groups of the form $T(2n,(2q+1)n)$ for any $q,n\in\mathbb{N}$. The only further obstruction is an algebraic one in showing that such modules are non-free, which we predict should be achievable in a similar way. The above verification for the $T(10,15)$ case should motivate this claim. 
\end{remark}

We will conclude by exhibiting new examples of non-free stably free modules over Baumslag-Solitar groups with standard presentation $$BS(m,n)=\langle a,b\mid ba^mb^{-1}=a^n\rangle$$

\begin{thmx}\label{20}
    The Baumslag-Solitar groups $BS(m,n)$ for $n=m+1$ or $n=m-1$  have a non-free stably-free $\mathbb{Z}BS(m,n)$-module of rank one.
\end{thmx}

We would like to note that in the above Theorem \ref{20}, $BS(m,n)$ is an aspherical group, and hence we can pose the following interesting question, which would generalise the results in \cite{MR2263062} over $BS(2,3)$.

\begin{question}
     Are the non-free stably-free $\mathbb{Z}BS(m,n)$-modules constructed in Theorem \ref{20} geometrically realisable?
\end{question}

In section \ref{gv} we explore a global view of Geometric realisation through filtrations of groups. In particular we prove the following obstruction to Theorem \ref{1} solving the Geometric realisation problem for general aspherical groups, which we restate equivalently using new notation as Theorem \ref{14}:

\begin{theorem}
    For $K$ the Klein bottle group, there is a stably-free $\mathbb{Z}K$-module that is not of the form of the modules in Theorem \ref{1}.
\end{theorem}

We would like to point the reader towards recent developments on the Geometric realisation problem (over aspherical groups) and related questions in the following papers. In Nicholson \cite{MR123}, the first examples of non-free stably-free $\mathbb{Z}G$-modules of arbitrary rank are constructed and also geometrically realised, resolving a problem from the 1979 problems list of C.T.C Wall. In Mannan \cite{MR456}, the first example of a non-trivial geometric realisation of a stably-free module over the Klein bottle group was constructed, building on the work of Harlander-Misseldine in \cite{MR2846158}.

\subsection*{Acknowledgements} I would primarily like to thank Johnny Nicholson, my supervisor for this project, for an enlightening introduction to this area of research, as well as many insightful discussions and much help and advice. I also would like to thank Jens Harlander for a very helpful correspondence and fellow student Lukas Ertl, also working under Johnny, for interesting conversation and insights. This work was completed as a UROP project at Imperial College London supported by an EPSRC vacation bursary, I would like to thank the university for making this possible.

\section{Preliminaries}

In this section we define two of the key objects of interest in this paper: relation modules and aspherical groups. We give a short overview of some relevant results in the literature on relation modules, and then see an equivalent formulation of the Geometric realisation problem for aspherical groups. We conclude by recalling a connection to Wall's $D2$ problem.

\subsection{Relation Modules}

Given a group presentation $P$ with generating set $\textit{\textbf{x}}$, we construct the presentation complex $K(P)$, a finite 2-complex, which has universal cover $\tilde{K}(P)$ which is also a 2-complex. The 1-skeleton $\tilde{K}(P)^{(1)}$ is called the Cayley graph $\Gamma(G,\textit{\textbf{x}})$ of $G$ on the generating set $\textit{\textbf{x}}$.

\begin{definition}\label{rm}
    The relation module associated with the pair $(G,\textbf{x})$ is the $\mathbb{Z}G$-module $M(G,\textbf{x}):=H_1(\Gamma(G,\textbf{x}))$, where the $\mathbb{Z}G$ action on homology is induced by the monodromy action of $G$ on the universal cover $\tilde{K}(P)$, see \cite[p3]{MR3752465}.
\end{definition}

Historically, relation modules have been of interest as they provided the first examples of non-free stably-free $\mathbb{Z}G$-modules in Berridge-Dunwoody \cite{MR540056} with $G$ the trefoil group. These are important as all free $\mathbb{Z}G$-modules are trivially geometrically realisable for $G$ aspherical, by taking $X_n=K\vee S^2\vee S^2...$ with $n$ copies of $S^2$, where $K$ is an aspherical 2-complex with fundamental group $G$, see Definition \ref{ad}. This space is a finite 2-complex and has $\pi_1(X_n)\cong G$ and $\pi_2(X_n)\cong (\mathbb{Z}G)^n$. As a result non-free stably-free $\mathbb{Z}G$-modules make the up all of the non-trivial cases of the geometric realisation problem when $G$ is aspherical by Proposition \ref{ga}.

The question of geometric realisation of the original modules constructed by Berridge-Dunwoody was open for over a decade, until corresponding finite 2-complexes were found by Preusser \cite{MR263}. She further generalised the approach of Berridge-Dunwoody to $T(2,2q+1)$ for any $q\in\mathbb{N}$ and $T(3,4)$, where $T(n,m)$. It was shown in this paper certain limitations of the methods in generalising to more complex torus knot groups. We have in fact managed to obtain a similar result over $T(10,15)$, a case not accessible to that paper.  

The general question of geometric realisation of relation modules over aspherical groups was answered in the affirmative in \cite[Corollary 4.3]{MR2263062}:

\begin{theorem}[Harlander-Jenson]
    If $G$ is aspherical and $M\in SF(\mathbb{Z}G)$ is a relation module, then $M$ is geometrically realisable.
\end{theorem}

To this day the vast majority of examples of finite 2-complexes with non-free $\pi_2$ and aspherical fundamental group have come from this source. Hence it seems likely that this is a good place to start in order to further the understanding of geometric realisation, and construct more examples. 

\subsection{The Geometric realisation problem and aspherical groups}\label{grp}
We will now try to understand the consequences of the restriction to `aspherical' groups:

\begin{definition}\label{ad}
    A finite 2-complex $K$ is called aspherical if $\pi_2(K)=0$, and a group $G$ is called aspherical if there is $K$ aspherical with $\pi_1(K)\cong G$. 
\end{definition}

This gives the following simplification of the Geometric realisation problem:

\begin{definition}
      Let $G$ be a group and $M\in SF(\mathbb{Z}G)$, we say that $M$ is geometrically realisable if there is a finite 2-complex $X$ such that $\pi_1(X)\cong G$ and $\pi_2(X)\cong M$. 
\end{definition}

\begin{proposition}\label{ga}
    For $G$ aspherical, the Geometric realisation problem is equivalent to whether every stably-free $\mathbb{Z}G$-module is geometrically realisable.
\end{proposition}

We will first recall the following result from \cite{MR33519} and establish some corollaries:

\begin{lemma}[Maclane-Whitehead]\label{M}
Algebraic $(G,2)$-complexes are chain homotopy equivalent if and only if their algebraic 2-types are isomorphic.
    \end{lemma}

A direct application of this is seen in \cite[Corollary 1.9]{MR3752465}:

\begin{corollary}\label{che}
    Let $A_1$ and $A_2$ be two algebraic $(G,2)$-complexes. Assume that $H^3(G,M)=0$ for all $\mathbb{Z}G$-modules $M$. Then $A_1$ and $A_2$ are chain homotopy equivalent if and only if $H_2(A_1)$ and $H_2(A_2)$ are isomorphic.
\end{corollary}

\begin{corollary}
    Let $G$ be an aspherical group and let $A$ be an algebraic $(G,2)$-complex with $H_2(A)\cong M$. If there is a finite 2-complex $X$ such that $\pi_2(X)\cong M$, then $A$ is geometrically realisable by $X$.
\end{corollary}

\begin{proof}
    By standard result, $G$ aspherical implies $H^3(G,M)=0$ for all $\mathbb{Z}G$-modules $M$. Hence it suffices to see that by Hurewicz's theorem, $H_2(\tilde{X})\cong \pi_2(\tilde{X})\cong\pi_2(X)\cong H_2(A)$, so by Lemma \ref{M} the cellular chain complex $C_\ast(\tilde{X})$ is chain homotopy equivalent to $A$, as $H_2(C_\ast(\tilde{X}))=H_2(\tilde{X})\cong H_2(A)$.
\end{proof}

\begin{proof}[Proof of Proposition \ref{ga}]

By results on syzygies, if $X$ is a finite 2-complex with $\pi_1(X)\cong G$ aspherical, we have that $\pi_2(X)=M$ is a stably-free $\mathbb{Z}G$-module. It suffices to see that every stably-free $\mathbb{Z}G$-module can be realised as $H_2(A)$ for some algebraic 2-complex $A$, which is proven in \cite[Proposition 8.18]{MR2012779}.
\end{proof}

The Geometric realisation problem is very important for a large part due to it's connection with a famous problem due to Wall \cite{MR171284}:

\begin{d2} Let $X$ be a connected finite $3$-complex of cohomological dimension $2$ with $\pi_1(X)\cong G$ and $H_3(\tilde{X},\mathbb{Z})=0$. Is $X$ homotopy equivalent to a finite $2$-complex?
\end{d2}

By the work of Johnson \cite{MR2012779} and Mannan\cite{MR2496351}, we have the following result:

\begin{theorem}[D2-problem and Geometric realisation] For $G$ a finitely presented group, the $D2$-problem and the Geometric realisation problem are equivalent.

\end{theorem}

\section{Construction of finite 3-complex $K$}

We initially follow closely the methods from the proof of \cite[Theorem 4.2]{MR2263062} to construct a finite 3-complex with certain algebraic properties. In the next section we use a result from homotopy theory to give general conditions under which this 3-complex will in fact be homotopy equivalent to a finite 2-complex. These results will be combined later to prove Theorems \ref{1} and \ref{2}. We first establish some topological prerequisites:

 \begin{lemma}\label{50}
     If $(X,A)$ is a CW pair, $Y$ is a CW-complex and $f:A\to Y$ is a cellular map, then the gluing $X\sqcup_f Y$ is a CW-complex
 \end{lemma}

 \begin{proof}

    To see this we take the cells of $X\sqcup_f Y$ to be the union of cells of $X$ not in $A$ and the cells of $Y$, and one can check that this gives a valid CW-complex structure.
 \end{proof}

\begin{definition}
    Given topological spaces $X, Y, Z$ and continuous maps $f:Z\to X$, $g:Z\to Y$, we define the double mapping cylinder $M(f,g)=X\sqcup_f (Z\times[0,1])\sqcup_g Y$. We write $X\cup (Z\times [0,1])\cup Y$ when the maps are implicit. The double mapping cylinder fits into a diagram:
 \[ 
    \begin{tikzcd}
Z \arrow[r, "f"] \arrow[d, "g"'] & Y \arrow[d, "i_Y",hook'] \\
X \arrow[r, "i_X"',hook]              & M(f,g)                 
\end{tikzcd}
\] which commutes up to homotopy equivalence, called the homotopy pushout of $f$, $g$. The maps $i_Y$, $i_Z$ are the natural inclusion maps, and will be omitted in future examples.
\end{definition} 

\begin{lemma}\label{dmc}
    Double mapping cylinders exist in the category of CW-complexes with cellular maps
\end{lemma}

\begin{proof}
    Apply the proof of lemma \ref{35} to each end of $Z\times[0,1]$.
\end{proof}

Let $H\leqslant G$ be aspherical groups. We first consider the case that we are given $K_H$ a finite 2-complex with $\pi_1(K_H)\cong H$, where $\pi_2(K_H)\in SF(\mathbb{Z}H)$ is a stably-free $\mathbb{Z}H$-module by results on syzgyies. We will show how to realise $\mathbb{Z}G\otimes_{\mathbb{Z}H} \pi_2(K_H)\in SF(\mathbb{Z}G)$ as the second homotopy group of a finite 3-complex with fundamental group $G$ in a natural way.

Let $L$ be an arbitrary aspherical 2-complexes with fundamental group $H$. It is easy to see that there is a cellular map $f:L\to K_{H}$ induced by the identity map of groups $H\to H$, for example by considering that the obstruction is trivial as $L$ is a  2-complex. Let $J$ be the aspherical 2-complex with fundamental group $G$, then likewise there is a cellular map $g: L\to J$ induced from the group homomorphism $H\hookrightarrow G$.

We form the following double mapping cylinder:

$$K_{G}:=K_{H}\cup (L\times [0,1])\cup J$$ where the gluing maps are the $f$ and $g$ described above. This fits into the following homotopy commutative diagram:

\[\begin{tikzcd}
L \arrow[r, "f"] \arrow[d, "g"'] & K_{H} \arrow[d,hook'] \\
J \arrow[r,hook]                            & {K_{G}}   
\end{tikzcd}\]

\begin{lemma}\label{5} $K_{G}$ is a finite 3-complexes with the following algebraic characteristics:

     i) $\pi_1(K_{G})\cong G$

     ii) $\pi_2(K_{G})\cong \mathbb{Z}G\otimes_{\mathbb{Z}H}\pi_2(K_{H})$ 
\end{lemma}

\begin{proof}
That $K_{G}$ is a finite 3-complex is clear from Lemma \ref{dmc} and the fact that $J$, $K_{H}$ are finite 2-complexes and $L\times[0,1]$ is a finite 3-complex, and all cells in $K_{G}$ are from these three spaces by the proof of Lemma \ref{50}.
    The fundamental group follows from Van Kampen's theorem on CW-complexes to get $\pi_1(K_{G})\cong G\ast_H H\cong G$, as the maps are the standard identity and inclusion.

    It remains to prove the result on the second homotopy module. 
    We follow the method of \cite{MR2263062} in the proof of Theorem $4.2$, recalling that $K_{G}=K_{H}\cup (L\times [0,1])\cup J$
    
    Let $X$ be the universal cover of $K_{G}$, with $p:X\to K_{G}$ the covering map. Let $K_u=K_{H}\cup (L\times [0,\frac{1}{2}])$, and $K_v=J\cup (L\times [\frac{1}{2},1])$, so $K_u$ and $K_v$ cover $K_{G}$. 
    
    Note that $K_u\cap K_v=L\times \{\frac{1}{2}\}$. Let $X_u$ and $X_v$ be some connected component of $p^{-1}(K_u)$ and $p^{-1}(K_v)$ respectively. We have that $Y=X_u\cap X_v$ is a component of $p^{-1}(K_u\cap K_v)=p^{-1}(L\times \{\frac{1}{2}\})$, so is homeomorphic to the universal cover of $L$.
    
    We follow the method outlined in the proof of \cite[Theorem 4.2]{MR2263062} applied to the simplicial tree associated with the covering of $X$ by $p^{-1}(K_u)$ and $p^{-1}(K_v)$. This yields transversals $T(G/H)$ and $T'(G/H)$, and the following long exact sequence of homology as a consequence of the Mayor-Vietoris spectral sequence:

    $$...\to \bigoplus_{g'\in T'(G/H)} H_2(g'Y)\to \bigoplus_{g\in T(G/H)} H_2(gX_u)\oplus H_2(X_v)\to H_2(X)\to \bigoplus_{g'\in T'(G/H)} H_1(g'Y) \to...$$
    
    Now see that $Y$ is contractible as $L$ is aspherical and $Y$ is homeomorphic to it's universal cover, so we get that

    $$\bigoplus_{g\in T(G/H)} H_2(gX_u)\oplus H_2(X_v)\cong H_2(X)$$

    We conclude by applying Hurwitz theorem: as $X_u$ is the universal cover of $K_u$ which is homotopy equivalent to $K_{H}$, we get $H_2(X_u)\cong\pi_2(X_u)\cong\pi_2(K_{H})$ and likewise $H_2(X_v)\cong\pi_2(J)=0$ as $J$ aspherical. Also, $H_2(X)\cong \pi_2(X)\cong\pi_2(K_{G})$ as $X$ is the universal cover of $K_{G}$ by definition. Finally as $$\bigoplus_{g\in T(G/H)} H_2(gX_u)\cong \mathbb{Z}G\otimes_H\pi_2(K_{H})$$ we have that 
    \[\pi_2(K_{G})\cong\mathbb{Z}G\otimes_{\mathbb{Z}H}\pi_2(K_{H}). 
        \qedhere\]
    
\end{proof}

Now fix an aspherical group $G$, and finite 2-complexes $K_1$, $K_2$ with $\pi_1(K_i)\cong G$ for $i=1,2$. Similar to before we will show how to realise $\pi_2(K_1)\oplus \pi_2(K_2)\in SF(\mathbb{Z}G)$ as the second homotopy group of a finite 3-complex in a natural way.

We form the double mapping cylinder $$K=K_{1}\cup (J\times [0,1])\cup K_{2}$$ where $J$ is an aspherical 2-complex with fundamental group $G$ the glueing maps $h_i:J\hookrightarrow K_{i}$ are induced by $G\to G$ identity map of fundamental groups. This fits into the following homotopy commutative diagram of a homotopy pushout:

\[\begin{tikzcd}
J \arrow[r, "h_1"] \arrow[d, "h_2"'] & {K_{1}} \arrow[d, hook'] \\
{K_{2}} \arrow[r, hook]            & K                         
\end{tikzcd}\]

\begin{lemma}\label{6}
    $K$ has the structure of a finite 3-complex with the following algebraic characteristics:

    i) $\pi_1(K)\cong G$

    ii) $\pi_2(K)\cong\pi_2(K_{1})\oplus\pi_2(K_{2})$
\end{lemma}

\begin{proof}
    The proof is analogous to that of Lemma \ref{5}, for the sake of brevity we will give only a sketch of the approach.

    We have fundamental group $G=G*_G G$ under the identity maps, and define $X$, $K_u$, $K_v$, $X_u$, $X_v$ and $Y$ by exchanging each space from the proof of lemma \ref{5} with the new space that appears in same corner of the corresponding homotopy commutative diagram in Lemma \ref{6}. This yields a similar long exact sequence, but in this case all the transversals $T(G/G)$ are trivial, so in fact the situation is simpler:

    $$...\to  H_2(g'Y)\to  H_2(gX_u)\oplus H_2(X_v)\to H_2(X)\to  H_1(g'Y) \to...$$ 
    
    Likewise, $Y$ is contractible, so the sequence collapses to give

    $$ H_2(gX_u)\oplus H_2(X_v)\cong H_2(X)$$
    
    Now, by following the new definitions of $X_u$ and $X_v$ and by application of Hurwitz's theorem, we have that $H_2(X_u)\cong\pi_2(K_{1})$, $H_2(X_v)\cong \pi_2(K_{2})$ and $H_2(X)\cong \pi_2(K) $, so 

    $$\pi_2(K)\cong\pi_2(K_{1})\oplus\pi_2(K_{2})$$ and we have the result. \qedhere

\end{proof}

\section{Relating standard pushouts and homotopy pushouts}\label{po}

In this section we will use a result from homotopy theory to give conditions under which standard and homotopy pushouts are homotopy equivalent. We then apply this to Lemmas \ref{5} and \ref{6} to show that under certain assumptions the finite $3$-complexes constructed are homotopy equivalent to finite 2-complexes.

\begin{definition}
    Given topological spaces $X, Y, Z$ and continuous maps $f:Z\to X$, $g:Z\to Y$, the pushout of $f,g$ in the category of topological spaces is defined to be $X\sqcup Y/\sim$ where for $x\in X, y\in Y$, $x\sim y$ if and only if there is $z\in Z$ such that $x=f(z)$ and $y=g(z)$. Write this as $X\cup Z\cup Y$. The pushout $P(f,g)$ fits into the following commutative diagram: 
    \[\begin{tikzcd}
X \arrow[r, "f"] \arrow[d, "g"'] & Y \arrow[d, "i_Y'"] \\
Z \arrow[r, "i_Z'"']             & {P(f,g)}           
\end{tikzcd}\]
where the maps $i_Z'$, $i_Y'$ are not necessarily inclusion maps (unlike in a homotopy pushout).
    
\end{definition}

Note that we have been considering $M(f,g)=X\cup (Z\times [0,1])\cup Y$, which is the double mapping cylinder formed as the homotopy pushout of $f$, $g$, which always exists in the category of CW-complexes with cellular maps. One might hope that the standard pushout $P(f,g)=X\cup Z\cup Y$ exists in the category of CW-complexes, and further that the pushout of finite 2-complexes is also a finite 2-complex that is homotopy equivalent to the homotopy pushout, but neither of these are true in general.

\begin{proposition}\label{36}
    Given CW-complexes $X,Y,Z$ and cellular maps $f:Z\hookrightarrow X$ and $g:Z\to Y$ such that $f$ is an inclusion of subcomplexes, we have that $M(f,g)$ is homotopy equivalent to $P(f,g)$.
\end{proposition}

To prove this we require the following lemmas:

\begin{lemma}\label{37}
    Given topological spaces $X,Y,Z$ and maps $f:Z\hookrightarrow X$ and $g:Z\to Y$ such that $f$ is a cofibration, we have that $M(f,g)$ is homotopy equivalent to $P(f,g)$.

\end{lemma}

\begin{proof}
   This is \cite[Proposition 6.49]{MR2839990}.
\end{proof}

\begin{lemma}\label{7}
If $X$ is a CW-complex and
 $A$ is a subcomplex of $X$, then the inclusion $A\hookrightarrow X$ is a cofibration
\end{lemma}

\begin{proof}
    This is a standard result, which follows from the fact that $S^k\hookrightarrow D^{k+1}$ is a cofibration, and that cofibrations are invariant under pushouts. 
\end{proof}

\begin{proof}[Proof of Proposition \ref{36}]
As $f:Z\hookrightarrow X $ is an inclusion of subcomplexes, it is a cofibration by Lemma \ref{7}, and hence $M(f,g)$ is homotopy equivalent to $P(f,g)$ by lemma \ref{37}.
\end{proof}

\begin{corollary}\label{44}
    Let $H\leqslant G$ be aspherical groups and $M\in SF(\mathbb{Z}H)$ be geometrically realisable by a finite 2-complex $K_H$ with aspherical subcomplex $L$ with fundamental group $H$, then \[\mathbb{Z}G\otimes_{\mathbb{Z}H}M\in SF(\mathbb{Z}G)\] is geometrically realisable.
\end{corollary}

\begin{proof}
    By Lemma \ref{5}, $\mathbb{Z}G\otimes_{\mathbb{Z}H}M\in SF(\mathbb{Z}G)$ is realised as the second homotopy group of a (finite 3-complex) double mapping cylinder $M(f,g)$ where $f:L\hookrightarrow K_H$ and $g:L \to J$ and $J$ is an aspherical 2-complex for $G$. In particular the induced map $f$ is an inclusion of subcomplexes, so by Proposition \ref{36} $M(f,g)$ is homotopy equivalent to the standard pushout $P(f,g)$. But this is equal to the glueing $K_G\sqcup_f J$ as in Lemma \ref{50} where $f$ is defined on the subcomplex $L$ of $K_G$, which is clearly a finite 2-complex.
\end{proof}

\begin{corollary}\label{45}
    Let $G$ be an aspherical group and $M,N\in SF(\mathbb{Z}G)$ be geometrically realisable by finite 2-complexes $K_{1}$ and $K_{2}$ containing aspherical subcomplexes $J_1$ and $J_2$ respectively with fundamental group $G$. Then \[M\oplus  N\in SF(\mathbb{Z}G)\] is geometrically realisable.
\end{corollary}

\begin{proof}
    We first note that we may assume that $J_1=J_2=J$ as otherwise we can take the induced map $f:J_1\to J_2$ and form $K_1\sqcup_f J_2$ which doesn't change $\pi_2$ by the proof of Corollary \ref{44} with $H=G$, and now contains aspherical subcomplex $J_2$.

    Now it suffices to note that for $h_i:J\hookrightarrow K_i$, $i=1,2$ maps induced by inclusions of subcomplexes, we have that $M(h_1,h_2)$ is homotopy equivalent to the finite 2-complex $P(h_1,h_2)$ by Proposition \ref{36}.
\end{proof}

\section{Proofs of Theorems \ref{1} and \ref{2}}

In this section we prove Theorem \ref{2} and then Theorem \ref{1}. The main idea of the first proof will be to use Corollaries \ref{44} and \ref{45} in succession. The second proof only requires that we verify that relation modules are geometrically realised by finite 2-complexes satisfying the conditions of Theorem \ref{2}.\\

\noindent\textbf{Theorem B.} \textit{If $H_1,...H_k\leqslant G$ for $k\geqslant 1$ are aspherical groups and for $i=1,...,k$ $N_i\in SF(\mathbb{Z}H_i)$ are geometrically realisable by finite 2-complexes $K_i$ containing aspherical subcomplexes $L_i$ with fundamental group $H_i$, then \[(\mathbb{Z}G\otimes_{\mathbb{Z}H_1} N_1)\oplus (\mathbb{Z}G\otimes_{\mathbb{Z}H_2} N_2)\oplus...\oplus(\mathbb{Z}G\otimes_{\mathbb{Z}H_k} N_k)\in SF(\mathbb{Z}G)\] is geometrically realisable.}

\begin{proof}
    First we let $f_i:L_i\to J$ for $J$ an aspherical 2-complex with fundamental group $G$ be the map induced by inclusion of fundamental groups $H_i\hookrightarrow G$. We define the pushout $\tilde{K_i}=K_i\sqcup_{f_i}J$  as in the proof of Corollary \ref{44}, with $L_i$ a subcomplex of $K_i$ by assumption. Then by Corollary \ref{44} these pushouts are precisely the geometric realisations of $\mathbb{Z}G\otimes_{\mathbb{Z}H_i}N_i\in SF(\mathbb{Z}G)$ for $i=1,...,k$. Further, these pushouts contain an aspherical subcomplex $J$ with fundamental group group $G$.

     To prove the direct sums are realisable, we proceed by induction on $n\leqslant k-1$, so assume that $(\mathbb{Z}G\otimes_{\mathbb{Z}H_1} N_1)\oplus (\mathbb{Z}G\otimes_{\mathbb{Z}H_2} N_2)\oplus...\oplus(\mathbb{Z}G\otimes_{\mathbb{Z}H_{n}} N_{n})$ is geometrically realisable by a finite 2-complex $R_{1}$ with aspherical subcomplex for $G$ given by some $J$ (this is true by assumption for the base case $n=1$). Now, $\mathbb{Z}G\otimes_{\mathbb{Z}H_{n+1}}M_{n+1}$ is also geometrically realisable by a finite 2-complex $R_{2}=\tilde{K_n}$ with aspherical subcomplex for $G$ given by $J$ as the choice made in the construction of $\tilde{K_i}$ was arbitrary. By Corollary \ref{45} with $J_1=J_2=J$, $((\mathbb{Z}G\otimes_{\mathbb{Z}H_1} N_1 )\oplus...\oplus(\mathbb{Z}G\otimes_{\mathbb{Z}H_{n}} N_{n}))\oplus(\mathbb{Z}G\otimes_{\mathbb{Z}H_{n+1}} M_{n+1})$ is geometrically realisable by some $K$. Further, it contains an aspherical subcomplex for $G$ as the construction from the proof of Corollary \ref{45} will give $K=P(h_1,h_2)$ for $h_i:J\hookrightarrow R_i$ the inclusion of  $J$ as a subcomplex for $i=1,2$ , which in particular preserves $J$ as a subcomplex of $K$. By induction we conclude the result. \qedhere

\end{proof}

As a corollary to this we prove:\\

\noindent\textbf{Theorem A.} \textit{If $H_1,...H_k\leqslant G$ for $k\geqslant 1$ are aspherical groups and $M_i\in SF(\mathbb{Z}H_i)$ for $i=1,...,k$ are relation modules, then \[(\mathbb{Z}G\otimes_{\mathbb{Z}H_1} M_1)\oplus (\mathbb{Z}G\otimes_{\mathbb{Z}H_2} M_2)\oplus...\oplus(\mathbb{Z}G\otimes_{\mathbb{Z}H_k} M_k)\in SF(\mathbb{Z}G)\] is geometrically realisable.}

\begin{proof}
    By \cite[Corollary 4.3]{MR2263062}, relation modules $M_i\in SF(\mathbb{Z}H_i)$ for $i=1,..,k$ are geometrically realisable by 
    
    $$K_i=(L_i\cup (W_i\times [0,1])\cup L_i)^{(2)}$$ respectively, for some carefully chosen $W_i$ associated with the relation module and $L_i$ aspherical 2-complexes with $\pi_1(L_i)\cong H_i$. This gives that relation modules satisfy the conditions of Theorem \ref{2}, so we have the result.  \qedhere

\end{proof}

\section{Applications}

Theorem \ref{2} and its corollaries yield a wealth of geometrically realisable stably-free modules over integral group rings of aspherical groups, and it is to be expected that a large amount of these new examples will be non-free.

As an application of Corollary \ref{4} we will exhibit an example of a new finite 2-complex with non-free $\pi_2$ and fundamental group $T(10,15)$, a torus knot group over which no such examples have been found. The method will also provide a framework to compute other such examples in a very similar way.

\begin{remark}
    Technichally, the standard torus knot groups arising as the fundamental group of a torus knot complement are only defined as $T(n,m)$ for $n,m\in\mathbb{Z}$ coprime. Without this coprimality assumption there is no obvious interpretation of the group in the same way, but it is still an interesting group to study in these contexts and in fact we have seen examples in the literature of these groups still being called 'torus knot groups',
\end{remark}

We initially let $H=T(2,3)$ be the trefoil group, with standard presentation 

$$\langle a,b \mid a^2=b^3\rangle$$

and $G=T(10,15)$, a (generalised) torus knot group with standard presentation

$$\langle x,y\mid x^{10}=y^{15}\rangle$$

We repeat the above presentations to define the elements of the integral group rings $\mathbb{Z}H$ and $\mathbb{Z}G$. Note that there is an embedding of groups $H\hookrightarrow G$ by the map $a\mapsto x^5$ and $b\mapsto y^5$, so we can consider $H\leqslant G$, and $\mathbb{Z}H\hookrightarrow\mathbb{Z}G$. 

\begin{remark}
    We can carry out the following method in the general setting $T(2,3)\hookrightarrow T(2n,3n)$ for any $n\in\mathbb{N}$, to get modules over $T(2n,3n)$. With some effort we predict that this will yield a similar result in most cases, although some care must be taken, in particular if $2|n$ or $3|n$ some modifications may be required. Further, we have $T(2,2q+1)\hookrightarrow T(2n,(2q+1)n)$ for any $q,n\in\mathbb{N}$, and so we can hopefully employ \cite{MR263} to get modules over $T(2n,(2q+1)n)$.
\end{remark}

By \cite{MR540056}, we have a rank one stably-free non-free (right) $\mathbb{Z}H$-relation module $M$, with generators 
$$\alpha'=(a^{-1}b^{-2}+a^{-1}ba^{-1}+a^{-1}b-b^{-4}-b^{-4}a-b^{-4}a^2-1,\ast)$$ and $$\beta'=(a^{-2}+a^{-1}b^{-1}a^{-1}-a^{-1}b^{-2}-b^{-2}a^{-1}-a^{-3}+b^{-4},\ast)$$ inside $(\mathbb{Z}H)^2$, where the second coordinate of the generators will not be important.

We will consider the module $\mathbb{Z}G\otimes_{\mathbb{Z}H}M\in SF(\mathbb{Z}G)$, which is geometrically realisable by Theorem \ref{1} as $M$ is a relation module. The majority of this section will focus on the proof of the following theorem:

\begin{theorem}\label{8}
    $N=\mathbb{Z}G\otimes_{\mathbb{Z}H}M$ is a stably-free non-free $\mathbb{Z}G$-module.
\end{theorem}

First note that by direct computation $N$ is generated by the images of $\alpha$ and $\beta$ in the inclusion $\mathbb{Z}H\hookrightarrow\mathbb{Z}G$, let these be denoted

$$\alpha=(x^{-5}y^{-10}+x^{-5}y^5x^{-5}+x^{-5}y^5-y^{-20}-y^{-20}x^5-y^{-20}x^{10}-1,\ast)$$ and $$\beta=(x^{-10}+x^{-5}y^{-5}x^{-5}-x^{-5}y^{-10}-y^{-10}x^{-5}-x^{-15}+y^{-20},\ast).$$

In \cite{MR540056} it is shown that $\alpha',\beta'\in(\mathbb{Z}H)^2$ cannot be generated by a single element of $\mathbb{Z}H$, which implies $M$ is non-free as it has rank one as a stably-free module. Likewise we will show the same property for $\alpha,\beta\in(\mathbb{Z}G)^2$, confirming that $N$ is not-stably free as also has rank one. 

\begin{remark}
    To show that $\alpha, \beta$ cannot be generated by a single element as a right $\mathbb{Z}G$-module, it suffices to check the same property for the first coordinate projection. By abuse of notation from now on we will let $\alpha,\beta$ denote this projection.
\end{remark}

    In \cite{MR540056} they consider the quotient $H\twoheadrightarrow H/H''$, and investigate $\mathbb{Z}H/H''\otimes_{\mathbb{Z}H}M\in SF(\mathbb{Z}H/H'')$, which they show to be non-free using nice arithmetic properties of the ring $\mathbb{Z}H/H''$, confirming the same property for $M\in SF(\mathbb{Z}H)$.

    In our case we note that there is a surjection $G\twoheadrightarrow H$ via the map $x\mapsto a$, $y\mapsto b$, which is a group homomorphism as $a^2=b^3\implies a^{10}=b^{15}$, which gives a surjection $G\twoheadrightarrow H/H''$. This allows us to work over the same ring $\mathbb{Z}H/H''$ as above, which is a `skew Laurent polynomial ring'. This can be seen via the following presentation of $H/H''$ exploited in \cite{MR540056}:

    $$\langle v,p,q\mid pq=qp, q=v^{-1}pv, v^{-1}qv=qp^{-1}, vpv^{-1}=q^{-1}p, vqv^{-1}=p\rangle$$

    in which the group element $v$ acts as the `variable', and the elements $p,q\in\mathbb{Z}(H/H'')'$ generate the coefficients.

    \begin{lemma}\label{bd}
We have the following definitions and properties of the ring $\mathbb{Z}H/H''$:
\begin{enumerate}
    \item For any $w\in\mathbb{Z}H/H''$, we can write $w=v^ma_m+...v^{m+n}a_{n+m}$ for $a_i\in\mathbb{Z}(H/H'')'$ where $n,m$ and $a_i$ for $i=m,...,n+m$ are unique.
    \item We may assume $a_m,a_{n+m}\neq 0$, call $n$ the degree of $w$.

    \item If $a_m,a_{n+m}\in H/H''$, call $w$ monic.
    \item If $a\mid b$ (on the left or right) in $\mathbb{Z}H/H''$ for any $b$ monic, then $a$ must be monic.
    \item For any $a,b\in\mathbb{Z}H/H''$ such that $deg(a)\geqslant deg(b)$ and $b$ monic, we have a Euclidean algorithm: there exists $q,r\in\mathbb{Z}H/H''$ such that $a=bq+r$ and $r=0$ or $deg(r)< deg(b)$.
    
\end{enumerate}

\end{lemma}

\begin{proof}
    This is lifted from Berridge-Dunwoody \cite{MR540056}.
\end{proof}

    Later we will work with $\mathbb{Z}H/H''\otimes_{\mathbb{Z}G}N$, generated by the images of $\alpha$ and $\beta$ under the surjective map $\mathbb{Z}G\twoheadrightarrow\mathbb{Z}H/H''$, but first we will first try to ascertain some arithmetic information from the further quotient $H/H''\twoheadrightarrow H/H'\cong\mathbb{Z}$. We define the natural map 

    $$\pi:\mathbb{Z}H/H''\to\mathbb{Z}H/H'\cong\mathbb{Z}[x,x^{-1}]$$ by $p,q\mapsto 1$, $v\mapsto x$. We will now formalise normalizing elements of $\mathbb{Z}[x,x^{-1}]$ as an element of $\mathbb{Z}[x]$ with non-zero constant term:
\begin{definition}
    Let $p\in\mathbb{Z}[x,x^{-1}]$, define the $\mathbb{Z}[x]$-realisation of $p$, $\tilde{p}$, to be the unique $x^kp\in\mathbb{Z}[x]$ such that the constant term of $x^kp$ is non-zero.    
\end{definition}

\begin{lemma}\label{21}
    If $a,b\in\mathbb{Z}[x]$ with $a=\tilde{a}$, $b=\tilde{b}$ and $a|b$ in $\mathbb{Z}[x,x^{-1}]$, $a|b$ in $\mathbb{Z}[x]$.
\end{lemma}
\begin{proof}
    $a|b$ in $\mathbb{Z}[x,x^{-1}]$ implies there is $ r\in\mathbb{Z}[x,x^{-1}]$ such that $ar=b$. It suffices to check that $r\in\mathbb{Z}[x]$, which is clear as the minimal non-zero coefficient of $a\in\mathbb{Z}[x]$ is the constant coefficient, so if $r$ has non-zero coefficient of a negative power of $x$, then so would the product $ar$, which is contradiction as $ar=b\in\mathbb{Z}[x]$ by assumption.
\end{proof}
 \begin{lemma}\label{22}
     Given $p,q\in\mathbb{Z}[x]$ with non-zero constant coefficients,
     if $(p,q)=(r)\triangleleft\mathbb{Z}[x,x^{-1}]$ is a principal ideal, we have that $gcd(p,q)|r$ in $\mathbb{Z}[x]$ (we may assume $r\in\mathbb{Z}[x]$).
 \end{lemma}                                                                \begin{proof}
     $(p,q)=(r)$ implies there exist $a,b\in\mathbb{Z}[x,x^{-1}]$ such that $ap+bq=r$. Hence $\tilde{a}p+\tilde{b}q=x^nr$ for some $n\in\mathbb{N}\cup\{0\}$. Consequently $gcd(p,q)|x^nr$ in $\mathbb{Z}[x]$, so as $x\nmid p,q$,  we have $gcd(p,q)|r$ in $\mathbb{Z}[x]$.
 \end{proof}                                                              
 \begin{corollary}\label{53}
     If $u,v\in\mathbb{Z}H/H''$ such that $(u,v)=(w)\triangleleft\mathbb{Z}H/H''$ is a right principal ideal, then \[gcd(\tilde{\pi(u)},\tilde{\pi(v)})=\tilde{\pi(w)}\]
     where $\pi$ is the map $\mathbb{Z}H/H''\to\mathbb{Z}H/H'\cong\mathbb{Z}[x,x^{-1}]$ seen before.
 \end{corollary}
\begin{proof}
    If $w|u,v$ (on the left) in $\mathbb{Z}G$, we have $\tilde{\pi(w)}|\tilde{\pi(u)},\tilde{\pi(v)}$ in $\mathbb{Z}[x,x^{-1}]$ as $x$ is a unit in $\mathbb{Z}[x,x^{-1}]$. Hence by Lemma \ref{21} $\tilde{\pi(w)}|\tilde{\pi(u)},\tilde{\pi(v)}$ in $\mathbb{Z}[x]$, so $\tilde{\pi(w)}|gcd(\tilde{\pi(u)},\tilde{\pi(v)})$ in $\mathbb{Z}[x]$. 
    Further $(u,v)=(w)$ implies $ (\tilde{\pi(u)},\tilde{\pi(v)})=(\tilde{\pi(w)})\triangleleft \mathbb{Z}[x,x^{-1}]$, so by Lemma \ref{22} 
    $gcd(\tilde{\pi(u)},\tilde{\pi(v)})|\tilde{\pi(w)}$ and we have the result.
\end{proof}

\begin{lemma}\label{ca}
    The images of $\alpha,\beta\in\mathbb{Z}G$ under the map $\mathbb{Z}G\twoheadrightarrow\mathbb{Z}H/H''$ are

$$\psi=-1+v^{-5}p^2-v^{-10}q^{-1}p^{-1}+v^{-20}p^{-2}q-v^{-25}q^2p^{-2}+v^{-35}p+v^{-40}q^{-1}p$$ and

$$\phi=1-v^{-5}(p^2+p^{-2})+v^{-10}(q^{-1}p+q^3p^{-2})-v^{-15}p^{-1}q$$ respectively, which are the generators of $\mathbb{Z}H/H''\otimes_{\mathbb{Z}G}N$.
    
\end{lemma}    

\begin{proof}
    The first part is a direct computation of the images of $\alpha$ and $\beta$, which we recall are the first coordinate projections of the generators of $N$. To see that these images generate $\mathbb{Z}H/H''\otimes_{\mathbb{Z}G}N$, we recall that in \cite{MR540056} it is seen that projection onto the first coordinate maps $\mathbb{Z}H/H''\otimes_{\mathbb{Z}G}N$ injectively into $ \mathbb{Z}H/H''$ as $\mathbb{Z}H/H''$ has no zero divisors, so $\psi,\phi$ are generators of $\mathbb{Z}H/H''\otimes_{\mathbb{Z}G}N$.
\end{proof}

 \begin{corollary}\label{ab}
        If $\alpha, \beta\in\mathbb{Z}G$ are generated by a single element $\gamma'$ of $\mathbb{Z}G$ as a right $\mathbb{Z}G$-module, then the image $\gamma$ of $\gamma'$ under $\mathbb{Z}G\twoheadrightarrow\mathbb{Z}H/H''$ generates $\mathbb{Z}H/H''\otimes_{\mathbb{Z}G}N$ as a right $\mathbb{Z}H/H''$-module and
  $\tilde{\pi(\gamma)}=\pm(x^{10}-x^5+1)$.
    \end{corollary}    

\begin{proof}
    If $\gamma'$ generates $\alpha$ and $\beta$ as a right $\mathbb{Z}G$-module, then $(\alpha,\beta)=(\gamma')$ is a right principal ideal, so $\mathbb{Z}H/H''\otimes_{\mathbb{Z}G}N=(\psi,\phi)=(\gamma)\triangleleft \mathbb{Z}H/H''$ by Lemma \ref{ca}. By Corollary \ref{53}, $gcd(\tilde{\pi(\psi)},\tilde{\pi(\phi)})=\tilde{\pi(\gamma)}$ and by direct computation of the left hand side of the equality we get the result, recalling that $\pi(p)=\pi(q)=1$ and $\pi(v)=x$.
\end{proof}

\begin{proof}[Proof of Theorem \ref{8}]

 By Lemma \ref{ca},  $\mathbb{Z}H/H''\otimes_{\mathbb{Z}G} N$ is generated as a right $\mathbb{Z}H/H''$-module by

$$\psi=-1+v^{-5}p^2-v^{-10}q^{-1}p^{-1}+v^{-20}p^{-2}q-v^{-25}q^2p^{-2}+v^{-35}p+v^{-40}q^{-1}p$$ and

$$\phi=1-v^{-5}(p^2+p^{-2})+v^{-10}(q^{-1}p+q^3p^{-2})-v^{-15}p^{-1}q$$
    
By Corollary \ref{ab} it suffices to show that $\gamma\in\mathbb{Z}H/H''$ cannot generate $\psi$ and $\phi$ with $\tilde{\pi(\gamma)}=x^{10}-x^5+1$. 

By Lemma \ref{bd} we can write $\gamma=v^ma_0+v^{m+1}a_1+...+v^{m+n}a_n$ for $a_i\in\mathbb{Z}(H/H'')'$, and we have that $\gamma$ is monic as the same holds for $\phi$. In particular $\gamma$ must have degree $10$ as a laurent polynomial, as $\pi(\gamma)$ does as well.

Now we note that in each step of the Euclidean algorithm in $\mathbb{Z}H/H''$ the only operation is killing the leading term of $a$:
if $a=a_0+...+v^ma_m$ and $b=b_0+...+v^n b_n$ with $m\geqslant n$, we let $q=v^{m-n} t$ such that $v^nb_nv^{m-n} t=v^ma_m$, which we can do as $b_n$ a unit by assumption. We then repeat this process with $a=a-bq$ and $b$ and continue until termination, as each time the degree of the term $a$ decreases. In particular if we apply the algorithm to $\psi$, $\phi$, as both have all terms of degree a multiple of $5$, if we write $\psi=\phi q+r$ for some $q,r\in\mathbb{Z}H/H''$ with $r=0$ or $deg(r)<deg(\phi)$, $r$ and $q$ must also only have terms of degree a power of $5$, as in each step $q=v^{-5n}h$ for some $h\in\mathbb{Z}(H/H'')'$.

 Now, if $\gamma$ generates $\psi$ and $\phi$ as a right $\mathbb{Z}H/H''$-module, it must also generate $r$ as a right $\mathbb{Z}H/H''$-module, so in particular $r$ must have degree $\geqslant 10$ and less than the degree of $\phi$ which is $15$. We conclude that $r=b_0+v^{-5}b_1+v^{-10}b_2$ for $b_i\in\mathbb{Z}(H/H'')'$, up to a multiplication by $v^n$. Note that as $r=\gamma u$ for some $u\in\mathbb{Z}H/H''$, $u$ must have degree $0$ so we may assume $\gamma=a_0+v^{-5}a_1+v^{-10}a_2$ for $a_0, a_2\in (H/H'')'$ and $a_1\in\mathbb{Z}(H/H'')'$, as $\gamma$ monic.

Further, we can see that $\pi(\phi)=1+2x^{-5}-2x^{-10}-x^{-15}$, so by inspection $\pi(\gamma)=1-x^{-5}+x^{-10}$, and if $\phi=\gamma w$ for some $w\in \mathbb{Z}H/H''$, i.e is generated by $\gamma$ as a right $\mathbb{Z}H/H''$-module, then $\pi(w)=1-x^{-5}$.

Now, we may assume $\gamma=1+v^{-5}a_1+v^{-10}a_2$ and $w=1-v^{-5}b$ (all other coefficients of $w$ must vanish by expanding out $\phi=\gamma w$) by dividing through by the constant terms, which are elements of $(H/H'')'$. By comparing coefficients in $\phi=\gamma w$, we have that $v^{-5}(a_1+b)=v^{-5}(p^2+p^{-2})$, so either $a_1=p^2$, $b=p^{-2}$ or $a_1=p^{-2}$, $b=p^2$, as $b$ is a unit. In each case we get $v^{-5}a_1v^{-5}b=q^2$ or $q^{-2}$, which gives a contradiction as this must appear in the coefficient of $v^{-5}$ in $\phi$ as the only other part of the coefficient is $v^{-10}a_2$ with $a_2$ a unit.

We conclude that $\mathbb{Z}H/H''\otimes_{\mathbb{Z}G}N$ cannot be generated by a single element as a right $\mathbb{Z}H/H''$-module, which confirms that $N$ is non-free, so we have the result. 
\end{proof}
Finally we prove:\\

\noindent\textbf{Theorem C.} \textit{There exist exotic presentations $\mathcal{P},\mathcal{Q}$ for the group $T(10,15)$ given by \[\mathcal{P}=\langle a,b \mid a^{10}=b^{15}, 1\rangle\]    \[\mathcal{Q}=\langle p,q, p',q'\mid p^{10}=q^{15}, p'^{2}=q'^3, p^{15}=p'^3, q^{20}=q'^4\rangle\] or equivalently there is a non-free stably-free $\mathbb{Z}T(10,15)$-module that is geometrically realisable.}

\begin{proof}
    Let $H=T(2,3)$ and $G=T(10,15)$. Recall that $N$ is geometrically realisable by Theorem \ref{1}, and is non-free by Theorem \ref{8}, so it remains to find the exotic presentations. Note that by standard result, after contracting the maximal spanning tree though the $0$-cells of $X$, which is a homotopy equivalence, we can assume $X$ is a presentation complex. We can in fact construct the required presentation as follows:

    By \cite[Theorem 4.5]{MR2263062}, there is a geometric realisation of $M\in SF(\mathbb{Z}H)$ by the presentation complex $\mathcal{X}_\mathcal{R}$ of $$\mathcal{R}=\langle x,y,x',y'\mid x^2=y^3, x'^2=y'^3, x^3=x'^3, y^4=y'^4\rangle $$ and we note that $\mathcal{X}_\mathcal{R}$ contains an aspherical subcomplex for $T(2,3)$ generated by the loops corresponding to $x,y$. We form the pushout $\mathcal{X}_\mathcal{R}\sqcup_f J$ as in the proof of Corollary \ref{44} along the map $f$ from the aspherical subcomplex of $\mathcal{X}_\mathcal{R}$ to the aspherical presentation complex $J$ of $\langle a,b\mid a^{10}=b^{15}\rangle$, induced by the inclusion of fundamental groups $H\hookrightarrow G$. This gives a finite 2-complex corresponding to the following presentation for $G$  $$\mathcal{Q}=\langle p,q, p',q'\mid p^{10}=q^{15}, p'^{2}=q'^3, p^{15}=p'^3, q^{20}=q'^4\rangle $$ with $\pi_2(\mathcal{X}_\mathcal{Q})\cong \mathbb{Z}G\otimes_{\mathbb{Z}H} M=N$ by Corollary \ref{44}.

    Now, if we take the standard presentation of $G$ and add a redundant generator we get the presentation $$\mathcal{P}=\langle a,b \mid a^{10}=b^{15}, 1\rangle$$ with presentation complex $\mathcal{X}_\mathcal{P}=J\vee S^2 $ for $J$ the aspherical presentation complex from above. Hence it has $\pi_2(\mathcal{X}_\mathcal{P})\cong \mathbb{Z}G$, which is not $Aut(G)$-isomorphic to $N$ (i.e for the identification of $G$ with $\pi_1(\mathcal{X}_\mathcal{P})$ and $\pi_1(\mathcal{X}_\mathcal{Q})$) as $N$ is a non-free $\mathbb{Z}G$-module by Theorem \ref{8}, so $\mathcal{X}_\mathcal{P},\mathcal{X}_\mathcal{Q}$ are not homotopy equivalent. Finally, the presentations $\mathcal{P},\mathcal{Q}$ both have deficiency $0$, so they are exotic presentations and we have the result.
\end{proof}

\section{Global view of geometric realisation}\label{gv}

We will consider filtrations of aspherical groups, by which we mean:

$${1}=G_0\leqslant G_1\leqslant ... \leqslant G_{n-1}\leqslant G_n=G$$

where $G_i$ aspherical for $i=0,...,n$. 

Let $\tilde{SF}(\mathbb{Z}G_i)$ be the commutative monoid formed by taking $SF(\mathbb{Z}G_i)/\sim$ where $M\sim N$ if and only if they are isomorphic as $\mathbb{Z}G_i$-modules.

Define $A_{G_i}$ to be the geometrically realisable subset of $\tilde{SF}(\mathbb{Z}G_i)$. Note that if the geometric realisation problem is true, we have that $A_{G_i}=SF(\mathbb{Z}G_i)$. 

Further, the inclusions $r_i:G_i\hookrightarrow G_{i+1}$ induce extension of scalars maps $${r_i}^\ast:SF(\mathbb{Z}G_i)\to SF(\mathbb{Z}G_{i+1})$$ defined by $M\mapsto \mathbb{Z}G_{i+1}\otimes_{\mathbb{Z}G_i}M$.

Hence the filtration of aspherical groups induces the following sequence of commutative monoids:

$$\{0\}=\tilde{SF}(\mathbb{Z}G_0)\xrightarrow[]{r_0^\ast} \tilde{SF}(\mathbb{Z}G_1) \xrightarrow[]{r_1^\ast}...\xrightarrow[]{r_{n-2}^\ast}\tilde{SF}(\mathbb{Z}G_{n-1})\xrightarrow[]{r_{n-1}^\ast} \tilde{SF}(\mathbb{Z}G_n)=\tilde{SF}(\mathbb{Z}G)$$.

The original aim of the development of the methods in this paper was to show that $A_{G_i}\leqslant SF(\mathbb{Z}G_i)$ was in fact a submonoid, and that we have $r_i^\ast(A_{G_i})\leqslant A_{G_{i+1}}$. This is equivalent to the claim that if $H\leqslant G$ are aspherical groups and $M,N\in SF(\mathbb{Z}H)$ are geometrically realisable, then $\mathbb{Z}G\otimes_{\mathbb{Z}H} M\oplus \mathbb{Z}G\otimes_{\mathbb{Z}H} N\in SF(\mathbb{Z}G)$ is geometrically realisable.

We have managed to prove the weaker result:  If $H_1,...H_k\leqslant G$ for $k\geqslant 1$ are aspherical groups and $M_i\in SF(\mathbb{Z}H_i)$ for $i=1,...,k$ are relation modules, then \[(\mathbb{Z}G\otimes_{\mathbb{Z}H_1} M_1)\oplus (\mathbb{Z}G\otimes_{\mathbb{Z}H_2} M_2)\oplus...\oplus(\mathbb{Z}G\otimes_{\mathbb{Z}H_k} M_k)\in SF(\mathbb{Z}G)\] is geometrically realisable.

We will consider the submonoid of $\tilde{SF}(\mathbb{Z}G)$ with the following generating set:

$$B_{G}=\langle\bigcup\limits_{j=0}^{n} \{\mathbb{Z}G\otimes_{\mathbb{Z}G_j}M: M\in \tilde{SF}(\mathbb{Z}G_j) \text{ is a relation module}
\}\rangle/\sim$$

Colloquially this means that we take the monoid generated by extensions of scalars of relation modules under the fixed filtration of groups, up to module isomorphism. 

By Theorem \ref{1} we know that $B_{G}$ is made up of (isomorphism classes of) geometrically realisable modules, we may ask whether it is possible that relation modules are in a significant way 'fundamental' to the question of geometric realisation. To be concrete: could we have $B_{G}=A_{G}$ or further that $B_G=\tilde{SF}(\mathbb{Z}G)$. We will prove that the above do not hold in general, by exhibiting the following counterexample.

\begin{theorem}\label{14}
    Let $K$ be the Klein bottle group, then there is a geometrically realisable $\mathbb{Z}K$-module $S$ such that $S$ is not generated by extensions of scalars of relation modules over any subgroup of $K$, and so in particular $S$ is not in $B_K$ for any filtration of the group $K$ and $B_K\neq A_K$. 
\end{theorem}

The following key lemmas will make this clear:

\begin{lemma}\label{12}
    If $H\leqslant K$ is a subgroup, then $H$ is abelian or isomorphic to a Klein bottle group.
\end{lemma}
\begin{proof}
    This is a standard result, but we will give a sketch of a proof. It can be seen that every galois covering of the Klein bottle is a surface. Infinite degree covering spaces will have cyclic fundamental group, and finite degree ones will be closed surfaces of Euler characteristic $0$. Hence they are Klein bottles or tori by classification of surfaces, so have abelian or Klein bottle group fundamental group.\qedhere
    
\end{proof}
\begin{lemma}\label{13}
   Relation modules over the Klein bottle group and abelian groups are free 
\end{lemma}

\begin{proof}
    The result for the klein bottle is due to Lars Louder \cite{MR309} from a general result, with a simpler proof in the case of the Klein bottle given in \cite[Theorem 5.1]{MR2846158}. Over abelian groups all stably-free modules are free so the result is trivial.
\end{proof}

\begin{corollary}\label{15}
    If $S\in SF(\mathbb{Z}K)$ is non-free and geometrically realisable, then $S\notin B_K$ for any choice of filtration of $K$.
\end{corollary}

\begin{proof}
    By Lemma \ref{12}, any filtration of $K$ must be formed out of abelian or Klein bottle groups, and by Lemma \ref{13}, all relation modules over these groups are free, so $B_K$ is generated by extensions of scalars of free modules, that are free. Hence $B_K$ consists only of free $\mathbb{Z}K$-modules, so $S\notin B_K$.
\end{proof}

We can now prove Theorem \ref{14}:

\begin{proof}
    It remains to see that by the result of Mannan \cite{MR456}, the Klein bottle group has a geometrically realisable stably-free non-free module, so the result follows by Corollary \ref{15}.
\end{proof}

Despite this concrete counterexample to the claim that $B_G=\tilde{SF}(\mathbb{Z}G)$ for all $G$ aspherical, we can still consider when the above might hold for a specific aspherical group, as in this case we would have confirmation of the geometric realisation problem and hence the D(2) problem fro that group, so we have the following result:

\begin{theorem}
    For $G$ aspherical, if $B_G=\tilde{SF}(\mathbb{Z}G)$, then $G$ has the D2 property, or equivalently the geometric realisation problem is true for $G$.
\end{theorem}

Hence it would be useful to develop tools to understand the failure of the above equality $B_G=\tilde{SF}(\mathbb{Z}G)$, and rather speculatively we would like to recommend a global approach to this problem by employing the aforementioned filtrations to consider the following sequences of monoids:

$$\{0\}=B_{G_0}\xrightarrow[]{r_0^\ast}B_{G_1}\xrightarrow[]{r_1^\ast}...\xrightarrow[]{r_{n-2}^\ast}B_{G_{n-1}}\xrightarrow[]{r_{n-1}^\ast}B_{G_n}=B_G$$

$$\{0\}=\tilde{SF}(\mathbb{Z}G_0)\xrightarrow[]{r_0^\ast} \tilde{SF}(\mathbb{Z}G_1) \xrightarrow[]{r_1^\ast}...\xrightarrow[]{r_{n-2}^\ast}\tilde{SF}(\mathbb{Z}G_{n-1})\xrightarrow[]{r_{n-1}^\ast} \tilde{SF}(\mathbb{Z}G_n)=\tilde{SF}(\mathbb{Z}G)$$

We hope that this will yield some information relating $B_G\leqslant\tilde{SF}(\mathbb{Z}G)$ to $B_{G_i}\leqslant\tilde{SF}(\mathbb{Z}G_i)$ for $i=0,...,n-1$.

\section{Stably-free modules over Baumslag-Solitar groups}
\maketitle

Recall that the Baumslag-Solitar groups are defined by the following standard presentations:

$$BS(m,n)=\langle a,b\mid ba^mb^{-1}=a^n\rangle$$\\

\noindent\textbf{Theorem D.} \textit{The Baumslag-Solitar groups $BS(m,n)$ for $n=m+1$ or $n=m-1$  have a non-free stably-free $\mathbb{Z}BS(m,n)$-module of rank one.}\\

We will prove the above result as a corollary of a more general theorem which will be stated in full generality after the necessary definitions have been set up.

The results here will build upon the general method of the Berridge-Dunwoody paper \cite{MR540056}, applying to more general cases through the result of Stafford \cite{MR782386} in which he establishes projective non-free right ideals over integral group rings of poly-infinite cyclic groups. Similar results of Artamanov \cite{MR4928011} are of help in clarifying technical details.

The following result has been lifted from the Stafford paper.

\begin{theorem}\label{18}
    Let $G$ be a non-abelian poly-infinite cyclic group, then $\mathbb{Z}G$ has a stably-free non-free right ideal $K$ of rank one.
\end{theorem}

 We will begin with the following definition for clarity:

\begin{definition}
    Given a non-abelian poly-infinite cyclic group $G$, we have a finite length subnormal series 
    $$G=G_n\triangleright G_{n-1}\triangleright...\triangleright G_0=\{1\}$$

    such that $G_i/G_{i-1}\cong\mathbb{Z}$.

    As $G$ is non-abelian, there is a minimal $j$ such that $G_j$ is non-abelian, then $G_{j-1}$ is abelian and by a standard splitting result $G_j\cong\mathbb{Z}\rtimes G_{j-1}$. Let $x\in G$ be a generator of the above factor of $\mathbb{Z}$ for any subnormal series, then we call $x$ a variable of $G$.
\end{definition}

The above definition is motivated by the fact that such `variable' elements play the role of a variable of a Laurent polynomial ring in the integral group ring $\mathbb{Z}G$.

\begin{lemma}\label{17}
    If $G$ is a torsion-free group and $G/N$ is a  metabelian torsion-free quotient with torsion-free abelianisation, $G/N$ is poly-infinite cyclic.
\end{lemma}

\begin{proof}
    Let $H=G/N$, then we have 
    
    $$H\triangleright H'\triangleright \{1\}$$

    with $H'\cong\mathbb{Z}^m$ abelian as $H$ is metabelian, and $H/H'\cong\mathbb{Z}^n$ as torsion-free by assumption. By the correspondence between subgroups of $H/H'$ and normal subgroups of $H$ containing $H'$, we can choose $H' \triangleleft H_1,...,H_n\triangleleft H$ such that $H_i/H'\cong\mathbb{Z}^i$, and further $H_{i-1}\triangleleft H_i$ by the same reasoning. 

    As $H'\cong\mathbb{Z}^m$, we have the required subnormal series, so 

    $$H=H_n\triangleright H_{n-1}\triangleright...\triangleright H_1\triangleright H_0=H'\triangleright K_{m-1}...\triangleright K_1\triangleright \{1\}$$

    and $H_i/H_{i-1}\cong\mathbb{Z}$, $K_j/K_{j-1}\cong\mathbb{Z}$ and finally $H_0/K_{m-1}\cong\mathbb{Z}$, so $H$ is a poly-infinite cyclic group.\qedhere
    \end{proof}

We now outline the setup required for the proof of Theorem \ref{20}:\\

Let $G$ be a group with $G/N$ a non-abelian poly-infinite cyclic quotient. By the result of Stafford there is a stably-free non-free right $\mathbb{Z}G/N$-module of rank one, $K$. For ease of notation let $H=G/N$ for the following remarks.

As $H$ is poly-infinite cyclic there is a finite length subnormal series 

$$H=H_n\triangleright H_{n-1}\triangleright...\triangleright H_1\triangleright H_0={1}$$

with $H_i/H_{i-1}\cong\mathbb{Z}$ and as noted earlier, $H$ non-abelian implies there is a minimal non-abelian term $H_j$. Recall that $H_j\cong \mathbb{Z}\rtimes H_{j-1}$, and let $x$ be a generator of the $\mathbb{Z}$ term.

\begin{remark}\label{24} 
Given the above element $x\in H$, the construction of $K$ in Theorem \ref{18} due to Stafford in \cite[Theorem 2.12]{MR782386} is defined by a single element $y\in H_{j-1}$ such that $x^{-1}yx\ne y$. Then we let $r=1+y$ or $r=1+y+y^3$ depending on certain properties of $y$, and we let $s=xrx^{-1}$. Artamonov \cite{MR4928011} gives $K\cong ker(\lambda)$ for $\lambda:\mathbb{Z}H^2\to\mathbb{Z}H$ given by $(1,0)\mapsto r$ and $(0,1)\mapsto x+s$ a right $\mathbb{Z}H$-module homomorphism. 
\end{remark}

\begin{definition}\label{35}
    For any $x,y\in H$, let $K(x,y)=ker(\lambda)$ for $\lambda:\mathbb{Z}H^2\to\mathbb{Z}H$ a right $\mathbb{Z}H$-module homomorphism given by $(1,0)\mapsto r$ and $(0,1)\mapsto x+s$ as defined above.
\end{definition}

We can now formulate the result of Stafford seen in Remark \ref{24} in the following lemma:

\begin{lemma}\label{27}
    If $x$ is a variable element of $H$ and $y\in H_{j-1}$ have $xyx^{-1}\ne y$, $K(x,y)$ is a stably-free non-free right $\mathbb{Z}H$-module.
\end{lemma}

In the notation of Berridge-Dunwoody \cite{MR540056}, we let $s_1=r$, $s_2=x+s$ and $w_1=sx^{-2}$, $w_2=x^{-1}-rx^{-2}$, and by that paper $K(x,y)$ is generated as a right $\mathbb{Z}H$-module by $(w_1 s_1-1, w_1 s_2)$ and $(w_2 s_1, w_2 s_2-1)$.

\begin{remark}\label{23}
    A key point made in \cite{MR540056} is that for a right module $K$ over any ring with generators as above, if $s_1w_1+s_2w_2=1$ we have that $K$ is stably-free.
\end{remark}

\begin{remark}\label{30}
Note that $s_1 w_1+s_2 w_2=rsx^{-2}+1-srx^{-2}-xrx^{-2}+sx^{-1}=1+rsx^{-2}-srx^{-2}=1$, so that algebraically this identity relies only on $rsx^{-2}=srx^{-2}$, so $r, s$ commuting, and $xrx^{-1}=s$. 
\end{remark}

The key idea of this paper is to induce $\mathbb{Z}G$-modules from $\mathbb{Z}G/N$-modules under the above conditions on $G/N$ to be poly-infinite cyclic, which is the initial method of Berridge-Dunwoody. We let $\tau:\mathbb{Z}G\twoheadrightarrow\mathbb{Z}G/N$ be the standard map.

 \begin{lemma}\label{26}
    Let $K(x,y)$ be a right $\mathbb{Z}G/N$-module satisfying the conditions of Lemma \ref{27}, and generated by $(w_1 s_1-1, w_1 s_2)$ and $(w_2 s_1, w_2 s_2-1)$.
    If $s_1', s_2', w_1', w_2'\in\mathbb{Z}G$ have $s_1' w_1'+ s_2' w_2'=1$, and $\tau(s_i')=s_i, \tau(w_i')=w_i\in\mathbb{Z}G/N$, the right $\mathbb{Z}G$-module $K'$ generated by $(w_1' s_1'-1, w_1' s_2')$ and $(w_2' s_1', w_2' s_2'-1)$ is stably-free non-free.
 \end{lemma}

 \begin{proof}
    That $K'$ is stably-free follows from Remark \ref{23}, so it remains to show that it is non-free. It suffices to see that $K(x,y)\cong \mathbb{Z}G/N\otimes_{\mathbb{Z}G} K'$, and hence as $K(x,y)$ non-free by Lemma \ref{27}, $K'$ non-free. 
 \end{proof}

\begin{theorem}\label{28}
    Let $G$ be a torsion-free group with $G/N$ a torsion-free metabelian quotient with torsion-free abelianisation. Say there exist $y',x'\in G$ with the following properties:\\
    $i)$ $\tau(x')=x\in G/N$ is a variable (recall above definition)\\
    $ii)$ $y'\in G'$ \\
    $iii)$ $y'$ and $x'y'x'^{-1}$ commute and have distinct image in $G/N$ \\

    then there is a non-free stably free module $\mathbb{Z}G$-module. 
\end{theorem}

\begin{proof}
    Let $r'=1+y'$ (if $r'=1+y'+y'^3$ the same method works) and $s'=1+x'y'x'^{-1}$. Now, if we set $s_1'=r'$, $s_2'=x'+s'$, $w_1'=s'x'^{-2}$ and $w_2'=x'^{-1}-r'x'^{-2}$, by Remark \ref{30} $s_1' w_1'+s_2' w_2'=1$ as $r'$ and $s'$ commute by assumption. Hence by remark \ref{23} these define a stably-free $\mathbb{Z}G$-module $K'$ with generators $(w_1' s_1'-1, w_1' s_2')$ and $(w_2' s_1', w_2' s_2'-1)$.

     We now see that $K(x,y)=\mathbb{Z}G/N\otimes_{\mathbb{Z}G} K'$ is generated by $(w_1 s_1-1, w_1 s_2)$ and $(w_2 s_1, w_2 s_2-1)$ as a right $\mathbb{Z}G/N$-module for $\tau(s_i')=s_i, \tau(w_i')=w_i\in\mathbb{Z}G/N$. We just need to check that $x$, $y$ satisfy the conditions of Lemma \ref{27}. 
    
    In particular $x$ is a variable by assumptions of Theorem \ref{28}, and $y\in G'/N$ means that it is contained in the largest abelian factor of the subnormal series of $G/N=H$ seen in the proof of Lemma \ref{17}. Finally $y'$ and $x'y'x'^{-1}$ have distinct images in $G/N$ by assumption implies $x$ and $y$ don't commute so by Lemma \ref{27} $K(x,y)$ is stably-free non-free. Hence we conclude by Lemma \ref{26} that the same holds for $K'$.\qedhere

\end{proof}

As a corollary to this we prove Theorem \ref{20}:

\begin{proof}

Recall the presentation of the Baumslag-Solitar groups $$BS(m,n)=\langle a,b\mid ba^mb^{-1}=a^n\rangle$$

    By standard result, there is a homomorphism $\phi: BS(m,n)\to GL_2(\mathbb{Q})$ such that $$\phi(a)=\begin{pmatrix}
        1& 1 \\ 0& 1
    \end{pmatrix}\ 
    \phi(b)=\begin{pmatrix}
        n/m&0\\0&1
    \end{pmatrix}$$
    
    It can be seen that $Im(\phi)$ is metabelian, so we let $Im(\phi)$ be our metabelian quotient $G/N$ in the assumptions of Theorem \ref{28}.

    The abelianisation of $BS(m,n)$ is
    $$BS(m,n)^{ab}\cong\mathbb{Z}\oplus\mathbb{Z}/(m-n)$$ so is torsion-free if and only if $n=m+1$ or $n=m-1$, from now on we will assume these conditions on $n$ and $m$.

    To see that $Im(\phi)^{ab}$ is torsion-free, we let $Im(\phi)=H$ and note that $BS(m,n)^{ab}\cong\mathbb{Z}$ surjects onto $H/H'$, so it suffices to check that $H/H'$ has a torsion-free element. But this is clear as $\phi(b)^k\notin H'$ for any $k\in\mathbb{N}$ as has determinant $(n/m)^k$. 

    Let $x'=b$ and $y'=a^m$ in the notation of the hypotheses of Theorem \ref{28}, we will check that all the conditions are satisfied:

    $i)$ To see that $\phi(b)$ is a variable, we consider the following subnormal series of $H=Im(\phi)$:

    $$H\triangleright H'\triangleright \{1\}$$

    with $H/H'\cong\mathbb{Z}$ and $H'\cong\mathbb{Z}$, so we just need that $\phi(b)\notin H'$ as $H'$ the largest non-abelian factor, but we have already seen this above, so $\phi(b)$ a variable.\\

    $ii)$ We need that $a^m\in BS(m,n)'$, it suffices to consider by the presentation that $a\in BS(m,n)'$ as $ba^mb^{-1}a{-m}=a^{n-m}=a$.\\

    $iii)$ $x'y'x'^{-1}=a^n$ and $y'=a^m$, so clearly these commute and are not equal in $Im(\phi)$ as $n=m+1$ or $n=m-1$ implies if $\phi(a^n)=\phi(a^m)$, $\phi(a)$ is the identity in $Im(\phi)$ which is a contradiction.\\

    As all the conditions of Theorem \ref{28} are satisfied, we have the result.\qedhere

    \end{proof}

    \begin{remark}
        We can further give the stably-free non free $\mathbb{Z}BS(m,n)$-module $K$ as the kernel of a $\mathbb{Z}BS(m,n)$-module homomorphism, where we recall how $K$ is constructed in Lemma \ref{26} applied to the module in Definition \ref{35}. In particular we let $r=1+y'=1+a^m$ and $s=x'y'x'^{-1}=a^n$, and

    $$\lambda: (\mathbb{Z}BS(m,n))^2\to\mathbb{Z}BS(m,n)$$ with $$(1,0)\mapsto 1+a^m, (0,1)\mapsto b+a^n$$

    gives $K=ker(\lambda)$.

    \end{remark}

\begin{remark}
    We believe that it would not be difficult to modify the above to obtain an infinite family of isomorphically distinct rank one stably-free modules over the same Baumslag-Solitar groups, the same was done over the Klein bottle group in\cite{MR2846158} based on \cite{MR4928011}, and their methods should extend to this case.
\end{remark}

Theorem \ref{20} yields algebraic $(G,2)$-complexes for $BS(m,m-1)$ by \cite[Proposition 8.18]{MR2012779}, that are chain-homotopically distinct from the trivial one built off the free module of rank $1$ by lemma \ref{che}, with same Euler characteristic, as $BS(m,m-1)$ is aspherical. 

\begin{question} Are such algebraic $(G,2)$-complexes for $G=BS(m,m-1)$ geometrically realisable or conversely do they give a solution to Wall's D2-problem?
\end{question}

Given the recent work of Mannan in \cite{MR2496351} who managed to geometrically realise modules over the Klein bottle group $K$ that were constructed following a similar method we have from \cite{MR2846158}, it seems likely that a similar attempt could be fruitful in geometrically realising our modules over Baumslag-Solitar groups.
\medskip

\bibliographystyle{plain}
\bibliography{bibliography.bib}
\end{document}